\documentclass{article}
\usepackage{graphicx} 

\usepackage[utf8]{inputenc}

\usepackage{amssymb}
\usepackage{latexsym}
\usepackage{exscale}

\usepackage[dvipsnames]{xcolor}

\usepackage{tikz}
\usetikzlibrary{trees}

\usepackage{amssymb, amsmath}
\usepackage{enumerate}
\usepackage{enumerate}
\usepackage[top=1in,left=1in,right=1in,bottom=1in]{geometry}
         \usepackage{ulem}
          \usepackage{xspace}
          \usepackage{cancel}

\usepackage{amssymb, amsmath}
\usepackage{enumerate}
\usepackage{enumerate}
\usepackage[top=1in,left=1in,right=1in,bottom=1in]{geometry}
\usepackage{ulem}

 \newcommand{\je}[1]{{\color{red}{#1}}}

\usepackage{verbatim}

\def \beq{\begin{equation}}
\def \eeq{\end{equation}}

\renewcommand{\rq}[1]{(\ref{#1})}

\newtheorem{prop}{Proposition}
\newtheorem{thm}{Theorem}
\newtheorem{cor}{Corollary}

\newcommand{\bR}{{ \mathbb R  }}
\newcommand{\bC}{\Bbb C}
\newcommand{\bZ}{\Bbb Z}
\newcommand{\bQ}{\Bbb Q}
\newcommand{\bN}{\mathbb{N}}
\newcommand{\bK}{\Bbb K}

\newcommand{\la}{\mbox{$\lambda$}}

\newcommand{\pa }{\partial }
\newcommand{\f}{\varphi}

\newcommand{\ep}{\epsilon}

\newcommand{\al}{\alpha }

\newcommand{\ga}{\gamma }

\newcommand{\La}{\Lambda }

\newcommand{\om}{{\omega}}

\def\<{\langle} \def\>{\rangle}

\title{Boundary null-controllability for the beam equation with classical
structural damping.}

\author{Sergei Avdonin and Julian Edward\footnote{Corresponding Author} }
\date{April 2026}

\begin{document}
\maketitle
{\bf Author affiliations}

Sergei Avdonin: Department of Mathematics and Statistics, University of Alaska Fairbanks, Fairbanks, 99775, Alaska, USA

Julian Edward: Department of Mathematics and Statistics, Florida International University, Miami, FLorida 33199, USA
            

\vspace{.5in}
             
\begin{abstract}
Let $\Delta$ be the  Dirichlet Laplacian on the interval $(0,\pi)$, and let $T>0$. 
We prove a well-posedness results for the structurally damped beam equation
$$u_{tt}+\Delta^2 u-\rho \Delta u_t=0, x\in (0,\pi),t>0$$
with various boundary conditions including
$$
u(0,t)=u_{xx}(0,t)=0; u(\pi,t)=f(t),u_{xx}(\pi,t)=0,
$$
and $f\in H_0^2(0,T)$
and appropriate initial conditions. 
Viewing $f$ as a control, we prove null controllability for all $\rho \leq 2$. For $\rho >2$, we show null controllability for arbitrary $T>0$ holds for almost all $\rho$, but fails for a dense subset of $(2,\infty)$.

An analagous result is proven for Neumann control. 
\end{abstract}


\section{Introduction}

Let $\Delta$ be the  Laplacian: $\Delta =\pa_x^2$ on the interval $(0,\pi)$. 
We will denote the Dirichlet Laplacian by $A_D=-\Delta$ 
with operator domain either $H^2(0,\pi)\cap H^1_0(0,\pi)$, 
and the Neumann Laplacian by $A_N=-\Delta$ with operator domain 
$\{ u\in H^2(0,\pi ): u'(0)=u'(\pi)=0\}.$
It is well known that $A_D,A_N$ are self-adjoint with non-negative spectra. In what follows, we will denote by $A_j$
the operator $A_D$ or $A_N$ when there is no need for distinction. Denote by $X_j^p$ the operator domain of $A_j^{p/2}$.

We will study boundary control problems for the equation
\beq 
u_{tt}+A_j^2 u+\rho A_j u_t=0,\ x\in (0,\pi), \ t>0, \ j= \mbox{D or N.}\label{beam0}
\eeq
with  positive constant $\rho .$
 Throughout this paper, controllability will always mean
the ability of steering any initial state $(u(x,0),u_t(x,0))$ to zero over a finite time by some appropriate
input function $f$ (i.e. exact controllability to zero or null controllability). 
The boundary controls will be either 
\begin{eqnarray}
u(0,t)=u_{xx}(0,t) & = & 0,\nonumber \\
u(\pi,t) & = & f(t),\nonumber\\ 
u_{xx}(\pi,t) & = & 0,\label{dir}
\end{eqnarray}
for the Dirichlet Laplacian, or
\begin{eqnarray}
u_x(0,t) =
u_x(\pi,t) & = & g(t),\nonumber\\ 
u_{xxx}(\pi,t)=u_{xxx}(0,t) & = & 0,\label{neu}
\end{eqnarray}
for the Neumann Laplacian. 

The term $A_j u_t$  models a specific dissipative effect, known as structural damping. To the best of our knowledge, this was introduced in \cite{CR}:
“The basic property of structural damping, which is said to be consistent with empirical studies,
is that the amplitudes of the normal modes of vibration are attenuated at rates which are proportional to the oscillation frequencies.” This model was also studied under the name “proportional
damping” (cf. \cite{Ba}). In this paper, we refer to the damping term $A_j u_t$ as ``classical" structural damping, but one can also consider 
$A_j^\al u_t$ for $\al \in (0,2]$.
The  case $\al=2$ is known as “Kelvin–Voigt” damping. For a distributive control
and the damping term being $A^{\al}u_t$,  $\al \in (0,2]$, this is the first class of parabolic-like control models considered
in \cite{LT},\cite{T2}, see also \cite{AL}. The case of boundary control for $\al <1$, where there are significant challenges to proving well-posedness, will be considered in upcoming work, \cite{AEI2} and \cite{AEI3}, also see \cite{E2}.
For boundary controllability in the undamped case in small time, $\rho =0$, see \cite{Z}.

\subsection{Main Results}

Consider the system 
\begin{eqnarray}
u_{tt}+A_D^2 u+\rho A_D u_t& = & 0,\label{beamD}\\
u(0,t)=u_{xx}(0,t) & = & 0,\label{bc19}\\
u(\pi,t) & = & f(t),\label{bc29}\\ 
u_{xx}(\pi,t) & = & 0,\label{bc39}\\
u(x,0)=u^0(x),\ u_t(x,0)& = & u^1(x).\label{initD}
\end{eqnarray}
The well-posedness and regularity of this system,  was considered by Triggiani for general bounded Euclidean domains, see \cite{T1}. Restricted to  our one dimensional case, his results show that $f\in L^2(0,T)$ would imply $u_t\in L^2(0,T;X^{-3/2-\ep})$ for any $\ep >0$. However, to state exact controllability results, we require $u_t$ to be continuous in $t$.  

Define $H_*^2(0,T)=\{ f\in H^2(0,T):\ f(0)=f'(0)=0\}$.
We prove the following: 
\begin{prop}
\label{wpD}
Let $T>0$ and $f\in H_*^2(0,T)$.
Let $u$ solve \rq{beamD}-\rq{initD}. 
Suppose $(u^0,u^1)\in X_D^3\times X^1_D$. 
Then 
$$
u(x,t)=\frac{x}{\pi}+v(x,t),
$$
with 
$$
v\in C(0,T;X^3_D)\cap C^1(0,T;X^1_D).
$$
Thus 
$$
u\in  C^1(0,T;X^0_D).
$$
\end{prop}
\begin{thm}\label{bcD}
\ 

A) Assume  $\rho \leq 2.$
Let $T>0.$ 
Given $(u^0,u^1)\in X^{3}_D\times X^{1}_D$, there exists $f\in H_0^2(0,T)$ such that the solution $u$ to the system  \rq{beamD}-\rq{initD} solves
$$u(x,T)=u_t(x,T)=0,
$$
and there exists a constant $C$ depending only on $\rho$ such that
\beq
\| f''\|_{L^2(0,T)}\leq Ce^{Q(T)}(\| u_0\|_{X^{3}_D}+\|u_1\|_{X^{1}_D}).\label{fe}
\eeq
Here $Q(T)\leq C'/T$, where the constant $C'$ depends on $\rho$.

\ 

B)  There exists a dense subset $S$ of $(2,\infty)$ such that $\rho \in S$ implies the system \rq{beamD}-\rq{initD}  is not approximately controllable for any time $T>0$.

\

C) For almost all $\rho \in (2,\infty)$,  the system  \rq{beamD}-\rq{initD} is null controllable for any time $T>0$, and \rq{fe} holds.

\

D) For any $T_*>0$, there exists $\rho >2$ such that the system  \rq{beamD}-\rq{initD} is null controllable for  $T>T_*$, and \rq{fe} holds,  but not for $T<T_*$.

\end{thm}
Here, the system would be approximately controllable in time $T$ if for all $(v^0,v^1)\in X^3_D\times X^1_D$, and all $\ep >0$, there exists $f\in H_0^2(0,T)$ such that 
$$
\| u(*,T)-v^0(*)\|_{X^3}+\|u_t(*,T)-v^1(*)\|_{X^1}<\ep .
$$

These results are all proven using the moment method. The striking transition from $\rho \leq 2$ to $\rho >2$ can be explained by associated frequency sets. For $\rho <2$, the frequencies are shown to satisfy the gap condition, whereas for $\rho > 2$ this need not be the case; parts C and D above are determined by the condensation index of the frequency sets, as studied in \cite{ABGT}.

We now consider the following system with Neumann control:
We discuss null-controllability for 
\begin{eqnarray}
u_{tt}+A_N^2 u+\rho A_N u_t& = & 0\label{beamn2}\\
u_x(0,t)=u_{x}(\pi ,t) & = & g(t)\label{bc1n2}\\
u_{xxx}(0,t)=u_{xxx}(\pi,t) & = & 0\label{bc3n2}\\
u(x,0)=u^0(x),\ u_t(x,0)& = & u^1(x).\label{initn2}
\end{eqnarray}

\begin{prop}\label{wpn}
Let $T>0$ and $g\in H_*^2(0,T)$.
Suppose $(u^0,u^1)\in X^4_N\times X^2_N$.  Then there exists a solution $u$ to our   system \rq{beamn2}-\rq{initn2}. Furthermore, 
$$
u(x,t)=xg(t)+v(x,t),
$$
with 
$$
v\in C(0,T;X^4_N)\cap C^1(0,T;X^2_N).
$$
Thus 
$$
u\in C^1(0,T;X^1_N).
$$
\end{prop}
Remark: the requirement that the same control is applied at both $x=0$ and $x=\pi$ is consistent with the assumption made by Trigianni, \cite{T1}, that 
$\int_{\Gamma}\pa_\eta u =0$ for $\Gamma$ the boundary of a  domain in $\bR^n$. It is well known that this condition is necessary for unique solvability of the associated Neumann problem; in our case that solution is
 $xg(t)$.

Now suppose $(u^0,u^1)\in X^4_N\times X^2_N$, and suppose there exists a null-control $g\in H_0^2(0,T)$ for \rq{beamn2}-\rq{initn2}. If we multiply $u$ by 1, respectively $t$, and integrate by parts over $(0,\pi)\times (0,T)$, we get
$$
\int_0^\pi u^1(x)dx=0, \mbox{ respectively }\int_0^\pi u^0(x)dx=0.
$$
In view of these necessary conditions for null controllability, 
 we restrict $A_N$ to the subspace orthogonal to the constant functions, with associated Sobolev spaces 
$$X^p_{N,0}=\{ v(x)\in X^p_N: \int_0^\pi v dx=0\}.$$
Such restrictions of domain are not uncommon in the literature, as we indicate next section.

\begin{thm}\label{bcN}
Assume  $\rho < 2.$
Let $T>0.$ 
Given $(u^0,u^1)\in X^{4}_{N,0}\times X^{2}_{N,0}$, there exists $g\in H_0^2(0,T)$ such that the solution $u$ to the system  \rq{beamn2}-\rq{initn2} solves
$$u(x,T)=u_t(x,T)=0,
$$
and there exists a constant $C$ depending only on $\al,\rho$ such that
$$\| g''\|_{L^2(0,T)}\leq Ce^{Q(T)}(\| u_0\|_{X^{4}_{N,0}}+\|u_1\|_{X^{2}_{N,0}}).$$
Here $Q(T)\leq C'/T$, where the constant $C'$ depends on $\rho$.

\end{thm}
The assumption $\rho <2$ is made for brevity of exposition; analogues of all parts of Theorem \ref{bcD} also hold for $\rho \geq 2$.

Theorem\ref{bcN} is  proven using the moment method with  estimates on biorthogonal functions found in \cite{AIS}.

This paper is organized as follows. In the next subsection, we compare our results with the literature. In Section 2.1, we prove the well-posedness in the Dirichlet case, and also null controllability for $\rho \leq 2$. The case $\rho >2$, which requires slightly different methods, is presented in Section 2.2. Then the results for Neumann well-posedness and null controllability are presents in Section 2.3. In Section \ref{app}, we recall a result on biorthogonal families used for $\rho\leq 2$.
\subsection{Discussion}

In this section, we discuss the relevant literature.

Regarding the well-posedness and regularity for when the boundary control is $L^2$,
the reader is referred to Triggiani's work in \cite{T},  which also discusses both Dirichlet and  Neumann control, in $N$ dimensions, and also discussed moment control, $u_{xx}(\pi, t)=f(t)$, and shear control, $u_{xx}(\pi, t)=f(t)$.

Regarding  boundary controllability results for the structurally damped beam equation, we are unaware of any for Dirichlet or Neumann control as addressed in this paper. 
Assuming $\rho <2$, there are several results in \cite{CR}, but in most examples the uncontrolled end is clamped.
For more discussion of controllability for a clamped beam, see \cite{AIS}. 
Several papers, \cite{H},
\cite{mil},  \cite{AIS}, all consider boundary controllability in higher dimensional settings, but all assume $\rho <2,$ and don't discuss Dirichlet or Neumann control. Miller, in
 \cite{mil},   considered boundary control on product spaces which, when restricted to the one dimensional setting, would be equivalent to 
$$
u(0,t)=u_{xx}(0,t)=u(\pi,t)=0, u_{xx}(\pi,t)=g(t).
$$
It should be noted that the primary focus in this paper is in obtaining precise estimates on the control cost function $e^{Q(T)}$ as $T\to 0^+$.  
The Dirichlet and Neumann boundary control in the higher dimensional case will be addressed by us in upcoming work, \cite{AEI3}.

The proofs of our results use the moment method, in which a key element is the associated exponential family in $L^2(0,T)$ having a  biorthogonal family of functions satisfying an appropriate estimate.
The construction of biorthogonal families originated, to the best of our knowledge, with \cite{H}, with refinements in \cite{AIS}, then \cite{AEI}, then \cite{BM}, \cite{BBM}, \cite{B}. To prove Theorem \ref{bcD} parts C,D, we use a construction appearing in \cite{ABGT}.

For results on interior control, where the restriction $\al=1$ is easier to relax, the reader is referred to \cite{LT},\cite{AL},  \cite{mil}, and finally \cite{AEI}. In \cite{mit}, controllability is proven via a Carleman estimate, assuming $\al=1$.
 
\section{Materials and Methods}

Let $\{ \om_n ,\f_n: n\geq 1\}$ be the eigenvalues and corresponding normalized eigenectors for $A_D$. Then we define, for any real $p$,
$$\|u\|_{X^p_D}^2=\sum_n n^{2p}|\< u,\f_n\> |^2.$$
Let $\{ \om_n ,\f_n: n\geq 0\}$ be the eigenvalues and corresponding normalized eigenectors for $A_N$, so $\f_0=1/\sqrt{\pi}$ and $\om_0=0$.  Then we define, for $p\geq 0$, 
$$\|u\|_{X^p_N}^2=|\< u,\f_0\> |^2+\sum_{n=1}^\infty n^{2p}|\< u,\f_n\> |^2.$$

\subsection{Dirichlet case}

We discuss null-controllability for 
\begin{eqnarray}
u_{tt}+A_D^2 u+\rho A_D u_t& = & 0\label{beamd1}\\
u(0,t)=u_{xx}(0,t) & = & 0\label{bc1d}\\
u(\pi,t) & = & f(t)\label{bc2d}\\ 
u_{xx}(\pi,t) & = & 0\label{bc3d}\\
u(x,0)=u^0(x),\ u_t(x,0)& = & u^1(x).\label{initd}
\end{eqnarray}

{\bf Proof of Proposition \ref{wpD}}

	We begin by using a classical trick to convert the boundary control problem to an interior control problem. 
	Let $U(x,t)=\frac{x}{\pi}f(t)$, and $v(x,t)=u(x,t)-U(x,t).$
	Then $v$ satisfies
	\begin{eqnarray}
		v_{tt}+\Delta^2 v+\rho(-\Delta) v_t & = & -f''(t)x/\pi,\label{v1}\\
		v(0,t) =v(\pi ,t)=v_{xx}(0,t)=v_{xx}(\pi ,t)& = & 0\\
		v(x,0) & =& u^0(x)-\frac{x}{\pi}f(0)\\
		v_t(x,0) & = & u^1(x)-\frac{x}{\pi}f'(0).\label{v4}
	\end{eqnarray}
By hypothesis, we have 
	$$
	f(0)=f'(0)=0.
	$$

	We write $v=v^0+v^f$, where $v^0$ is the solution of 
	\rq{v1}-\rq{v4} with $f=0$, and $v^f$ is the solution with $u^0=u^1=0.$
		We now obtain a formulas for $v^f$.       
        First, we note the following Fourier expansion
		$$
		\frac{x}{\pi}=\sum_1 x_m\sqrt{\frac{2}{\pi}}\sin ( m x)
		=:\frac{1}{\pi}\sum_1\frac{2(-1)^{m+1}}{m}\sin ( m x).$$ 
		Set $v^f=\sum_{n=1}^\infty a_n(t)\f_n(x)$. Then 
		\rq{v1} implies
		$$
		\sum (a_n''+\rho a_n'n^{2}+a_nn^4   )\f_n(x)=-\sum f''(t)x_n\f_n(x),
		$$
		hence
		\beq
		a_n''+\rho n^{2}a_n'+n^4a_n=-f''(t)x_n,\ a_n(0)=a_n'(0)=0,\ \forall n\in \bN.\label{an}
		\eeq
		In what follows, it will be convenient to 
		set
\beq 
\beta_n=\frac{-\rho n^{2}}{2}, \ 
		\al_n=n^2\frac{\sqrt{4-\rho^2}}{2},\mbox{ and }\la_n^{\pm}:=\beta_n\pm i\al_n.\label{la}
\eeq
We now assume $\rho \neq 2$, indicating at the end the modifications necessary for $\rho =2$.
		We  find $v^f(x,t)=\sum_na_n(t)\f_n(x)$ using variation of parameters in \rq{an}.
		We have the  Wronskian equalling
		$-2i\al_ne^{-2\beta_nt}$.
		Hence
		\begin{eqnarray}
			a_n(t)
			& =& \frac{x_n}{2i\al_{n}}
			\int_0^t\big ( f''(s)\big )
			(e^{\la_n^+(t-s)}-e^{\la_n^-(t-s)})ds,
			\ n\in \bN, \label{vf2} \\
            a_n'(t)
			& =& \frac{x_n}{2i\al_{n}}
			\int_0^t\big ( f''(s)\big )
			(\la_n^+e^{\la_n^+(t-s)}-\la_n^-e^{\la_n^-(t-s)})ds,
			\ n\in \bN. \label{vf}
		\end{eqnarray}
Recall
$$
|x_n|\asymp 1/n,\ |\al_n| \asymp n^2, |\la_n |\asymp n^2.
$$
Also, the sets $\{ \la_n^+: n\in \bN\}, 
\{ \la_n^+: n\in \bN\}$ each satisfy the gap condition, i.e. there exists $\ep >0$ such that 
$$\inf_n |\la_n^j -\la_m^j|\geq \ep, j=\pm , \ m\neq n.$$ Hence 
$\{ e^{\la_n^+ t}: n\in \bN\}$,  $\{ e^{\la_n^- t}: n\in \bN\}$ each form Bessell sequences in $L^2(0,T)$, and it follows that their union
 $\{ e^{\la_n^\pm t}: n\in \bN\}$ forms a Bessel sequence in $L^2(0,T)$, and furthermore the constants defining the Bessel inequality will be vary continuously in $T$ for $T>0$. 
Hence, from \rq{vf2},\rq{vf}, we deduce
$$
v^f\in C(0,T;X_D^3)\cap C^1(0,T;X_D^1).
$$
It is easy to see that 
$$
v^0\in C(0,T;X_D^3)\cap C^1(0,T;X_D^1), \mbox{ and }U \in C^1(0,T;X_D^0),
$$
so  Proposition \ref{wpD}, for $\rho \neq 2$,  follows from 
$u^f=v^f+v^0+U$.

In the case $\rho =2$, in which case 
\beq
\la_n^+=\la_n^-=-n^2, \forall n,\label{la2}
\eeq
a calculation very similar to the one above gives 
\begin{eqnarray}
			a_n(t)
			& =& {x_n}
			\int_0^t f''(s)
			(s-t)e^{n^2(s-t)}ds,
			\ n\in \bN, \label{vf2'} \\
            a_n'(t)
			& =& {x_n}
			\int_0^t f''(s)
			\big (
            -n^2(s-t)e^{n^2(s-t)}-e^{n^2(s-t)}\big )ds,
			\ n\in \bN. \label{vf'}
		\end{eqnarray}
It is known, see for instance \cite{AEI2}, that the family $\{ e^{n^2t},te^{n^2t}\}$ forms a Bessel sequence in $L^2(0,T)$, the proposition can be proven by adapting the argument for $\rho \neq 2$; the details are left to the reader. 
$\Box$

{\bf Remark} it is clear from \rq{vf} that $a_n'(t)$ has the same regularity as $f'(t)$, illustrating why regular controls are needed for the null controllability problem.

\

{\bf Proof of Theorem \ref{bcD}, part A}.
 We present the proof for $\rho <2$, leaving the similar proof for $\rho =2$ to the reader. 
We adopt the notation of proof of the proposition. Since we will construct $f\in H_0^2(0,T)$, we have 
$f(T)=f'(T)=0,$
so $U(x,T)=U_t(x,T)=0$.
	Thus null controllability is for \rq{beamd1}-\rq{initd} is equivalent to the existence of $f$ such that 
	\beq 
	v^f(x,T)=-v^0(x,T); \ v_t^f(x,T)=-v^0_t(x,T) .
	\label{cont2}
	\eeq
We will find an associated moment problem. We need to express the ``free wave", $v^0$, as a Fourier series. Suppose for $j=0,1$, the initial conditions have Fourier expansion $u_j(x)=\sum_{n=1}^\infty u_n^j\f_n(x)$. Then
		from the above, 
		$$v^0(x,t)=\sum_n\big ( c^1_ne^{\la_n^+t}+c^2_ne^{\la_n^-t}\big )\f_n(x),$$
		with 
		$$c_n^2=\frac{\la_n^+u_n^0-u_n^1}{\la_n^+-\la_n^-},
		c^1_n=u_0^n-c^2_n.
		$$
We use the following notation: 
		\beq \label{tc1}
		-v^0(x,T)=
		\sum_{n\in \bN}\big (- c^1_ne^{\la_n^+T}
		-c_n^2e^{\la_n^-T}  \big )
		\f_{n}(x)=:\sum_{n\in \bN}\ga_n^1\f_{n}(x),
		\eeq 
		\beq \label{tc2}
		-v^0_t(x,T)=\sum_{n\in \bN}\big (- c^1_n\la_n^+e^{\la_n^+T}
		-c_n^2\la_n^-e^{\la_n^-T}  \big )
		\f_{n}(x)=:\sum_{n\in \bN}\ga_n^2\f_{n}(x).
		\eeq

By \rq{cont2}, \rq{vf2},\rq{vf}, \rq{tc1},\rq{tc2}, we get 
\begin{eqnarray}
\ga_n^1 & =& \frac{x_n}{2i\al_n}\int_0^T\big ( f''(s)\big )(e^{\la_n^+(T-s)}-e^{\la_n^-(T-s)})ds,\ n\in \bN, \label{md1'} \\
 \ga_n^2& =& \frac{x_n}{2i\al_n}
			\int_0^T\big ( f''(s)\big )
			(\la_n^+e^{\la_n^+(T-s)}-\la_n^-e^{\la_n^-(T-s)})ds,
			\ n\in \bN. \label{md2'}
		\end{eqnarray}
Setting
\beq
\zeta_n^1=-\frac{\la_n^+\gamma^1_n-\ga_n^2}{x_n}, \ 
\zeta_n^2=-\frac{\la_n^-\gamma^1_n-\ga_n^2}{x_n},\label{z}
\eeq
we rewrite \rq{md1'},\rq{md2'} as 
\begin{eqnarray}
\zeta_n^1 & =& \int_0^T g''(s) e^{\la_n^+(T-s)}ds,\ n\in \bN, \label{md1} \\
 \zeta_n^2& =& 
			\int_0^T g''(s)
			e^{\la_n^-(T-s)}ds,
			\ n\in \bN. \label{md2}
		\end{eqnarray}

Requiring $f(T)=f'(T)=0$ implies that $f''$ also satisfies
\begin{eqnarray}
0 &= & \int_0^T f''(s)ds\label{md3} \\
0 &= & \int_0^T sf''(s)ds.\label{md4} 
\end{eqnarray}	
Our moment problem is thus \rq{md1}-\rq{md4}.
We now solve for $f''$, hence $f$. 
Set $\zeta^j_0=0$ for $j=1,2$. Let $\{ g_{0,2}(t)\}\cup \{ g_m(t): m\in \bZ\}$ be the associated biorthogonal family. Also, write 
$$\zeta_k
=\left \{ 
\begin{array}{cc}
\zeta_k^1,& k>0,\\
\zeta_k^2, & k<0.
\end{array}
\right .
$$

Then formally, 
$$
f''(t)=\sum_{k\in \bK}\zeta_kg_k(t).
$$
It is straightforward to show $|\zeta_k|=O(e^{-\rho  Tn^2/2}).$ Combining with the estimate \rq{boe} in the Section \ref{app}, we see that the series above converges in $L^2(0,T)$.

\subsection{Dirichlet control in the case $\rho >2$.}\label{para1}
In this subsection we prove parts B, C, D of Theorem \ref{bcD}. 

Setting 
\beq 
r=(\rho+\sqrt{\rho^2-4})/2,  \label{rr}
\eeq
we have  
\beq
\la_m^+= -\frac{1}{r} m^{2},\ \la_n^-= -rn^{2},\label{bad}
\eeq
and $r>1.$
Arguing as in the previous subsection, proving null-controllability is equivalent to solving the moment problem \rq{md1}-\rq{md4}. 

We first prove part B of Theorem \ref{bcD}.
A simple algebra exercise shows  $r\notin \bQ$ is equivalent to the set $\{ \la_m^+,\la_n^-: n,m\in \bN\}$ being simple. In particular, if $r\in \bQ$, 
 there exist $m\neq n$ such that $\la^+_m=\la_n^-$, in which case our moment problem is overdetermined.
 It is then easy to deduce that our system is not even approximately controllable. Since $\bQ$ is dense in $\bR$, it is easy deduced from \rq{rr} that the associated $\rho$ are dense in $(2,\infty)$, proving part B.

 We now prove parts C, D, adapting arguments used in  \cite{ABGT}. This will involve a biorthogonal set for the exponential family
$$
\{ e^{-n^2t/r}, e^{-n^2tr}: n\in \bN\}
$$
in the interval $L^2(0,T)$. We now assume $r\notin \bQ$.

Suppose $\{ \varkappa_n: n\in \bN\} =\{ m^2/r,m^2r: m\in \bN\}$, listed in increasing order. 
Define 
$$
E(z)=\prod_{n=1}^\infty(1-z^2/\varkappa_n^2).
$$
We define the condensation index:
$$
c(\Lambda)=\limsup_{n\to \infty }\frac{\ln (|E'(\varkappa_n)|)}{\varkappa_n}.
$$
Since 
$$
\sin (\pi z)=\pi z \prod_{n\in \bN}(1-z^2/n^2),
$$
we get
$$
E(z)=-\frac{1}{\pi^4z^2}\sin(\pi \sqrt{rz})
\sinh (\pi \sqrt{rz})\sin(\pi \sqrt{z/r})\sinh (\pi \sqrt{z/r}),
$$
where we choose the branch of square root satisfying $\sqrt{-1}=i$. We compute
\begin{eqnarray*}
|E'(|\la_n^+|)| & = & \left (\frac{r^3}{2\pi^3n^5}\sinh (\pi n)\sinh (\pi n/r)\right )|\sin (\pi n/r)|=:A_n|\sin (\pi n/r)|\\
|E'(|\la_n^-|)| & = & \left (\frac{1}{2r^3\pi^3n^5}\sinh (\pi n)\sinh (\pi rn) \right )|\sin (\pi rn)|=:B_n|\sin (\pi n/r)|.
\end{eqnarray*}
Following \cite{ABGT},
$$
\frac{r^3}{2\pi^3n^5}\sinh (\pi )\sinh (\pi /r)\leq A_n\leq \frac{r^3}{2\pi^3n ^5}e^{\pi n} e^{\pi n/r},
$$
\beq 
\frac{1}{2\pi^3n^5r^3}\sinh (\pi )\sinh (\pi r)\leq B_n\leq \frac{1}{2\pi^3n ^5r^3}e^{\pi n} e^{\pi nr}.\label{AB}
\eeq 

Labeling 
$$
c (\Lambda^\pm )=\limsup_{n\to \infty }\frac{\ln (1/|E'(|\la_n^\pm|)|)}{|\la_n^\pm|},
$$
we have $c(\Lambda)=\max (c (\Lambda^+ ),c (\Lambda^- ))$, where by \rq{AB},
$$
c(\Lambda^+)=\limsup_{n\to \infty }\frac{-\ln (|\sin(\pi n/r)|)}{n^2/r},c(\Lambda^-)=\limsup_{n\to \infty }\frac{-\ln (|\sin(\pi r n)|)}{rn^2}
$$
\begin{prop}
\ 

A) for almost all $\rho >2$, $c(\Lambda)=0$.

B) Let $\tau \in [0,\infty ]$, there exists $\rho >2$ such that $c(\Lambda)=\tau $. Furthermore, the set of $\rho >2$ such that
$c(\Lambda)=\tau $ is dense in $(2,\infty).$
\end{prop}
Proof: we refer heavily to \cite{ABGT}. 
In their notation, $c(\Lambda)=\max (l_1/r,l_2)$. 
Comparing their Equation 6.16 with our formulas, with $r=\sqrt{d}$, 
$$
l_1=c(\Lambda^+)/r,\ l_2=c(\Lambda^-).
$$
According to 
\cite{ABGT}, Proposition 6.24, for almost all $r\in (0,\infty)$, we have $l_1=l_2=0$. Part A follows from \rq{rr}

We now prove Part B. Let $\tau \in [0,\infty]$. By \cite{ABGT}, Corollary 6.25,  the set of $r$ such that $l_2=\tau$ is dense in $(0,\infty )$. Since we have $r>1$, it follows for such $r$ that $c(\Lambda)=\tau.$ $\Box$

Finally, we clarify the relationship between the condensation index and biorthogonal families.
The following result, which we rewrite to  adapt to our notation, is found in \cite{ABGT}, Remark 4.3,
\begin{thm}
Let $r>1$ and 
suppose $c(\Lambda)<\infty .$
Then there exists a family $\{ q_m:m\in \bN \}$ of functions biorthogonal in $L^2(0,T)$ to $\{ e^{\varkappa_n^2(T-t)r}: n\in \bN \}$.
For any $\varepsilon > 0$  there exist positive constants 
$C_{1,\varepsilon}, \, C_{2,\varepsilon} $ such that
\beq 
C_{2,\varepsilon} \, e^{(c(\Lambda) - \varepsilon)\sqrt{\lambda_m}}
\;\leq\;
\|q_m\|_{L^2(0,T;C)}
\;\leq\;
C_{1,\varepsilon} \, e^{(c(\Lambda) + \varepsilon)\sqrt{\lambda_m}},
\quad \forall m \in \bN.
\label{bio}
\eeq 
\end{thm}
Applying this to the moment problem \rq{md1}-\rq{md4}, we immediately get 
\begin{cor}
Let $T>c(\Lambda)$.
Let  $\rho >2.$
Given $(u^0,u^1)\in X^{3}\times X^{1}$, there exists $f\in H_0^2(0,T)$ such that the solution $u$ to the system \rq{beamD}-\rq{initD} solves
$$u(x,T)=u_t(x,T)=0,
$$
and there exists a constant $C$ depending only on $\rho$ such that
$$\| f''\|_{L^2(0,T)}\leq C(\| u^0\|_{X^{3}}+\|u^1\|_{X^{1}}).$$
\end{cor}

Parts C,D of Theorem \ref{bcD} follow immediately from this corollary.

\subsection{Well-posedness and control for Neumann case}

In this section, we assume for simplicity of presentation that
	$\al =1 $ and $\rho <2.$ We write $X^p$ for $X_N^p$ for simplicity.
Consider 
\begin{eqnarray*}
u_{tt}+A_N^2 u+\rho A_N u_t& = & 0\label{beamn1}\\
u_x(0,t)=u_{x}(\pi ,t) & = & g(t)\label{bc1n1}\\
u_{xxx}(0,t)=u_{xxx}(\pi,t) & = & 0\label{bc3n1}\\
u(x,0)=u^0(x),\ u_t(x,0)& = & u^1(x).\label{initn1}
\end{eqnarray*}

\noindent{\bf Proof of Proposition \ref{wpn}:}

	Let $U(x,t)=xg(t)$, and $v(x,t)=u(x,t)-U(x,t).$
	Then $v$ satisfies
	\begin{eqnarray}
		v_{tt}+\Delta^2 v+\rho (-\Delta) v_t & = & -g''(t)x,\label{vn1}\\
		v_x(0,t) =v_x(\pi ,t)=v_{xxx}(0,t)=v_{xxx}(\pi ,t)& = & 0\\
		v(x,0) & =& u^0(x)-xg(0)\\
		v_t(x,0) & = & u^1(x)-xg'(0).\label{vn4}
	\end{eqnarray}
By hypothesis, we have 
	$$
	g(0)=g'(0)=0.
	$$

	We write $v=v^0+v^g$, where $v^0$ is the solution of 
	\rq{vn1}-\rq{vn4} with $g=0$, and $v^g$ is the solution with $u^0=u^1=0.$ To solve for $v^g$,
	we first we note the following Fourier expansion
		$$
		x=x_0/\sqrt{\pi}+\sum_1 x_m\sqrt{\frac{2}{\pi}}\cos ( m x)
		:=\frac{\pi}{2}+\sqrt{\frac{2}{\pi}}\sum_1 \frac{(-1)^m-1}{m^2}\cos ( m x)
        .$$ 
In what follows, it will be convenient to denote the normalized eigenfunctions by 
		$\{ \f_n(x):n\geq 0\}$.
		Set $v^g=\sum a_n(t)\f_n(x)$. Then 
		\rq{vn1} implies
		$$
		\sum_0^\infty (a_n''+\rho a_n'n^{2}+a_nn^4   )\f_n(x)=-\sum_0^\infty g''(t)x_n\f_n(x),
		$$
		hence
		\beq
		a_0''(t)=-g''(t)x_0/\sqrt{\pi};\mbox{ and for } n\geq 1,\  
        a_n''+\rho n^{2}a_n'+n^4a_n=-g''(t)x_n,\ a_n(0)=a_n'(t)=0.\label{an}
		\eeq
		In what follows, it will be convenient to 
		set
		$$\beta_n=\frac{-\rho n^{2}}{2}, \ 
		\al_n=n^2\frac{\sqrt{4-\rho^2}}{2},\mbox{ so }\la_n^{\pm}=\beta_n\pm i\al_n.
		$$
		We next find $v^g(x,t)=\sum_na_n(t)\f_n(x)$ using variation of parameters in \rq{an}.
		We have the  Wronskian equalling
		$-2i\al_ne^{-2\beta_nt}$.
		Hence, for $n\geq 1$,
		\begin{eqnarray}
			a_n(t)
			& =& \frac{x_n}{2i\al_n}
			\int_0^t\big ( g''(s)\big )
			(e^{\la_n^+(t-s)}-e^{\la_n^-(t-s)})ds,
			\ n\in \bN, \label{vf2n} \\
            a_n'(t)
			& =& \frac{x_n}{2i\al_n}
			\int_0^t\big ( g''(s)\big )
			(\la_n^+e^{\la_n^+(t-s)}-\la_n^-e^{\la_n^-(t-s)})ds,
			\ n\in \bN\, \label{vfn}
		\end{eqnarray}
while
\beq
a_0(t)=-\frac{x_0}{\sqrt{\pi}}\int_{s=0}^t\int_{r=0}^sg''(r)dr; \ a_0'(t)=-\frac{x_0}{\sqrt{\pi}}\int_{r=0}^tg''(r)dr.\label{vf1n}
\eeq
Recall that 
$$
|x_n|\asymp 1/n^2,\ |\al_n| \asymp n^2\asymp |\la_n |, \mbox{ for } n \neq 0.
$$
Since $\{ e^{\la_n^\pm t}\}$ forms a Bessel sequence on $L^2(0,T)$, we get 
$$
t\mapsto (\{ n^4 a_n(t)\},\{ n^2a_n'(t)\})
$$
is continuous from $\bR^+$ to $\ell^2 \times \ell^2$. Hence
$$
v^g\in C(0,T;X^4)\cap C^1(0,T;X^2).
$$
It is easy to see that 
$$
v^0\in C(0,T;X^4)\cap C^1(0,T;X^2), \mbox{ and }U \in C^1(0,T;X^1),
$$
so  Proposition \ref{wpn}  follows from 
$u=v^g+v^0+U$.$\Box$

\ 

Proof of theorem: we adopt the notation of proof of the proposition. Our construction will have $g\in H_0^2(0,T)$, so 
$g(T)=g'(T)=0,$
so $U(x,T)=U_t(x,T)=0$.
Thus null controllability for \rq{beamn1}-\rq{initn1} is equivalent to the existence of $g$ such that 
	\beq 
	v^g(x,T)=-v^0(x,T); \ v_t^g(x,T)=-v^0_t(x,T) .
	\label{cont2n}
	\eeq
We derive the associated moment problem. First, we express the ``free wave", $v^0$, as a Fourier series. Suppose for $j=0,1$, the initial conditions have Fourier expansion $u^j(x)=\sum_{n=1}^\infty u_n^j\f_n(x)$. Then
		from the above, 
		$$v^0(x,t)=\sum_{n\geq 1}\big ( c^1_ne^{\la_n^+t}+c^2_ne^{\la_n^-t}\big )\f_n(x),$$
		with 
		$$c_n^2=\frac{\la_n^+u_n^0-u_n^1}{\la_n^+-\la_n^-},
		c^1_n=u_0^n-c^2_n, n\geq 1.
		$$
We use the following notation: 
		\beq \label{tc1n}
		-v^0(x,T)=-c_0^+\f_0(x)-c_0^-T\f_0(x)+
		\sum_{n\in \bN}\big (- c^1_ne^{\la_n^+T}
		-c_n^2e^{\la_n^-T}  \big )
		\f_{n}(x)=:\sum_{n\geq 0}\ga_n^1\f_{n}(x),
		\eeq 
		\beq \label{tc2n}
		-v^0_t(x,T)=-c_0^-T\f_0+\sum_{n\in \bN}\big (- c^1_n\la_n^+e^{\la_n^+T}
		-c_n^2\la_n^-e^{\la_n^-T}  \big )
		\f_{n}(x)=:\sum_{n\in \bN}\ga_n^2\f_{n}(x).
		\eeq

By \rq{cont2n}, \rq{vf2n},\rq{vfn},  \rq{tc1n}. \rq{tc2n}, we get 
\begin{eqnarray}
\ga_n^1 & =& \frac{x_n}{2i\al_n}\int_0^T\big ( g''(s)\big )(e^{\la_n^+(T-s)}-e^{\la_n^-(T-s)})ds,\ n\in \bN, \label{mn1'} \\
 \ga_n^2& =& \frac{x_n}{2i\al_n}
			\int_0^T\big ( g''(s)\big )
			(\la_n^+e^{\la_n^+(T-s)}-\la_n^-e^{\la_n^-(T-s)})ds,
			\ n\in \bN. \label{mn2'}
		\end{eqnarray}
With $\zeta_n^j$ given by \rq{z},
we rewrite \rq{mn1'},\rq{mn2'} as 
\begin{eqnarray}
\zeta_n^1 & =& \frac{x_n}{2i\al_n}\int_0^T g''(s)g )e^{\la_n^+(T-s)}ds,\ n\in \bN, \label{mn1} \\
 \zeta_n^2& =& \frac{x_n}{2i\al_n}
			\int_0^T g''(s)
			e^{\la_n^-(T-s)}ds,
			\ n\in \bN, \label{mn2}
		\end{eqnarray}       
while the requirements $g(T)=g'(T)=0$ imply
\begin{eqnarray}
0 &= & \int_0^T g''(s)ds,\label{mn3} \\
0 &= & \int_0^T sg''(s)ds.\label{mn4} 
\end{eqnarray}	

Our moment problem is thus \rq{mn1}-\rq{mn4}.
We can then solve for  $g''$, hence for $g$, mimicking the arguments used in the Dirichlet case; the details are left to the reader.

\subsection{Biorthogonal Functions}\label{app}

An important part of our proof of null controllability is to construct  suitable sets of biorthogonal functions associated to $\{ e^{i\la_n t}:n\in \bN\}$. For this, we will adapt arguments found in 
 \cite{AEI}, which is a slight extension of \cite{AIS}, which in turn generalizes a result in \cite{H}.  

Let $\La$ be a discrete subset $\bC^+$. 
We introduce a function $\nu :[0,\infty)\mapsto [0,\infty)$
which describes the density of $\Lambda$ 
$$ 
\#\{ \la_n\in \Lambda\setminus \{\la_k,\la_{-k}\}  :  |\la_n -\la_k| < s\} \leq \nu (s), \forall k.
$$
Suppose $\Lambda$ satisfies 
\beq
\nu(s)=0, \ s<R_0,\label{separ}
\eeq
 for a positive $R_0$, which is equivalent to $\La$ satisfying a uniform gap condition. 
Suppose further the following estimates hold: 
there exist positive constants, $C_1,C_2$  such that 
\beq \label{nu}
C_0r^{1/2}\leq \nu(r)\leq C_1r^{1/2},\ \forall r>2R_0.
\eeq 

Recall
$$\< f,g\>=\int_0^Tf(t)\overline{g(t)}dt,
$$
where the bar denotes complex conjugation.
\begin{prop}\label{prop2}
Let $T>0$. 
Given $\La^i=\{ \la_k: k\in \bZ\}.$
Suppose there exists a function $\nu (r)$ satisfying  the estimates   \ref{separ} and \ref{nu}.
 Then there exists a family of functions $\{  g_{m,j}(t); m\in \bK ,\ j=1,2 \} $ in $L^2(0,T)$ satisfying 
$$ \< g_{m,1},te^{i\la_kt}\>=0,\  \< g_{m,1},e^{i\la_kt}\> =\delta_{m,k},\ \< g_{m,2},te^{i\la_kt}\>=\delta_{m,k},\  \< g_{m,2},e^{i\la_kt}\> =0 .
$$
Furthermore, there exist positive constants $C_2,C_3$ depending  only on $R_0,T,C_0,C_1$ such for $j=1,2,$
\beq
\| g_{m,j}\|_{L^2(0,T)}\leq C_2 e^{C_3(\Im (\la_m))^{\varkappa}},\label{boe}
\eeq
where $\varkappa$ is defined in \ref{nu}.
\end{prop}
We apply this result as follows:

\noindent{\bf Case 1: $\rho <2$.} In this case,
we set
 $$
\la_k =
\left \{ 
\begin{array}{cc}
-i\la^+_k, & k>0,\\
-i\la^-_{|k|}, & k<0,
\end{array}\right .
 $$
We then define $\la_0=0$, and $\Lambda=\{ \la_m: m\in \bZ\}$. Then $\Lambda$ satisfies \rq{separ} and \rq{nu}. It follows that the set of functions in $L^2(0,T)$ 
$$\{ t\}\cup \{ e^{i\la_mt}: m\in \bZ\}$$
has an associated biorthogonal family
$$\{g_{0,2}(t)\}\cup \{ g_m(t): m\in \bZ\}$$
satisfying the \rq{boe}.

\noindent{\bf Case 2: $\rho =2$.}
In this case, recall $\la_n^-=\rho_n^+$ for all $n\geq 1$. Letting $\la_0=0$ and $\la_m=-i\la_m$ for $m\geq 1$,
we set $\La=\{  \la_m: m\geq 0\}$, and set of functions
$$ \{ e^{i\la_mt}: m\geq 0\}.$$
Then there exists a family of functions $\{  g_{m,j}(t); m\geq 0 ,\ j=1,2 \} $ in $L^2(0,T)$ satisfying 
$$ \< g_{m,1},te^{i\la_kt}\>=0,\  \< g_{m,1},e^{i\la_kt}\> =\delta_{m,k},\ \< g_{m,2},te^{i\la_kt}\>=\delta_{m,k},\  \< g_{m,2},e^{i\la_kt}\> =0 ,
$$
and satisfying \rq{boe}.

\section{Conclusions}
We considered the Euler-Bernouilli beam, for either Dirichlet or Neumann control applied at one end. To prove continuous dependance of the velocity on time, we assume the control and energy spaces are regular. For these spaces, we prove well-posedness. For small damping parameter, $\rho \leq 2$, we then prove null-controllability for Dirichlet control; for Neumann control, we must assume the energy spaces are orthogonal constants. For $\rho >2$, we apply number theoretic results to prove that 
null controllability holds in some cases, but not others. 

We list several possible extensions to this work. First, it should be possible to extend this study to higher dimensions. Second, one can consider the beam equation with a spectral fractional damping term, $A^\alpha u_t$, with $\al \in (0,1)$; this will be addressed in \cite{AEI2},\cite{AEI3}. Third, it is natural to consider the beam equation with non-constant coefficients, but this will require different methods. Finally, the methods in this work should also apply to higher order controls such as moment control, $u_{xx}(\pi,t)=f(t)$, and shear control, $u_{xxx}(\pi,t)=f(t)$.

\section{Author contributions}
Conceptualization, S.A. and J.E.; Methodology, S.A. and J.E.; Software, not applicable; Validation, S.A. and J.E.; Formal Analysis, S.A. and J.E.; Investigation, S.A. and J.E.; Resources, S.A. and J.E.; Data Curation, S.A. and J.E.; Writing – Original Draft Preparation, J.E.; Writing – Review and Editing, S.A. and J.E..; Visualization, not applicable; Supervision, S.A. and J.E.; Project Administration, S.A. and J.E.; Funding Acquisition, not applicable.

\section{Funding}

The research of S.A. was  supported  in part by the National Science Foundation, grant DMS 2308377, and by the Ministry of Education and Science of the Russian Federations part of the program of the Moscow Center for Fundamental and Applied Mathematics under the Agreement No. 075-15-2025-345.

\section{Acknowledgments}

Sergei Ivanov, who was a very active collaborator in our project of studying the structurally damped beam, passed away the past February. We expect he would have been a coauthor of this paper if he were alive.

\section{Conflicts of Interest}
The authors declare no conflicts of interest.

\end{document}